\UseRawInputEncoding
\documentclass[12pt]{amsart}
\usepackage{amssymb}
\usepackage[colorlinks=true]{hyperref}
\usepackage[all]{xy}
\usepackage[toc,page,title,titletoc,header]{appendix}
\usepackage{xcolor}

\begin{document}
\theoremstyle{plain}
\newtheorem{thm}{Theorem}[section]
\newtheorem*{thm1}{Theorem 1}
\newtheorem*{thm1.1}{Theorem 1.1}
\newtheorem*{thmM}{Main Theorem}
\newtheorem*{thmA}{Theorem A}
\newtheorem*{thm2}{Theorem 2}
\newtheorem{lemma}[thm]{Lemma}
\newtheorem{lem}[thm]{Lemma}
\newtheorem{cor}[thm]{Corollary}
\newtheorem{pro}[thm]{Proposition}
\newtheorem{propose}[thm]{Proposition}
\newtheorem{variant}[thm]{Variant}
\theoremstyle{definition}
\newtheorem{notations}[thm]{Notations}
\newtheorem{rem}[thm]{Remark}
\newtheorem{rmk}[thm]{Remark}
\newtheorem{rmks}[thm]{Remarks}
\newtheorem{defi}[thm]{Definition}
\newtheorem{exe}[thm]{Example}
\newtheorem{claim}[thm]{Claim}
\newtheorem{ass}[thm]{Assumption}
\newtheorem{prodefi}[thm]{Proposition-Definition}
\newtheorem{que}[thm]{Question}
\newtheorem{con}[thm]{Conjecture}
\newtheorem*{assa}{Assumption A}
\newtheorem*{algstate}{Algebraic form of Theorem \ref{thmstattrainv}}

\newtheorem*{dmlcon}{Dynamical Mordell-Lang Conjecture}
\newtheorem*{condml}{Dynamical Mordell-Lang Conjecture}
\newtheorem*{congb}{Geometric Bogomolov Conjecture}

\newtheorem*{pdd}{P(d)}
\newtheorem*{bfd}{BF(d)}

\newtheorem*{probreal}{Realization problems}
\numberwithin{equation}{section}
\newcounter{elno}                
\def\points{\list
{\hss\llap{\upshape{(\roman{elno})}}}{\usecounter{elno}}}
\let\endpoints=\endlist
\newcommand{\SH}{\rm SH}
\newcommand{\Tan}{\rm Tan}
\newcommand{\res}{\rm res}
\newcommand{\Om}{\Omega}
\newcommand{\om}{\omega}
\newcommand{\la}{\lambda}
\newcommand{\mc}{\mathcal}
\newcommand{\mb}{\mathbb}
\newcommand{\surj}{\twoheadrightarrow}
\newcommand{\inj}{\hookrightarrow}
\newcommand{\zar}{{\rm zar}}
\newcommand{\Exc}{{\rm Exc}}
\newcommand{\an}{{\rm an}}
\newcommand{\red}{{\rm red}}
\newcommand{\codim}{{\rm codim}}
\newcommand{\Supp}{{\rm Supp}}
\newcommand{\rank}{{\rm rank}}
\newcommand{\Ker}{{\rm Ker \ }}
\newcommand{\Pic}{{\rm Pic}}
\newcommand{\Der}{{\rm Der}}
\newcommand{\Div}{{\rm Div}}
\newcommand{\Hom}{{\rm Hom}}
\newcommand{\im}{{\rm im}}
\newcommand{\Spec}{{\rm Spec \,}}
\newcommand{\Nef}{{\rm Nef \,}}
\newcommand{\Frac}{{\rm Frac \,}}
\newcommand{\Sing}{{\rm Sing}}
\newcommand{\sing}{{\rm sing}}
\newcommand{\reg}{{\rm reg}}
\newcommand{\Char}{{\rm char\,}}
\newcommand{\Tr}{{\rm Tr}}
\newcommand{\ord}{{\rm ord}}
\newcommand{\id}{{\rm id}}
\newcommand{\NE}{{\rm NE}}
\newcommand{\Gal}{{\rm Gal}}
\newcommand{\Min}{{\rm Min \ }}
\newcommand{\Max}{{\rm Max \ }}
\newcommand{\Alb}{{\rm Alb}\,}
\newcommand{\GL}{{\rm GL}\,}        
\newcommand{\PGL}{{\rm PGL}\,}
\newcommand{\Bir}{{\rm Bir}}
\newcommand{\Aut}{{\rm Aut}}
\newcommand{\End}{{\rm End}}
\newcommand{\Per}{{\rm Per}\,}
\newcommand{\ie}{{\it i.e.\/},\ }
\newcommand{\niso}{\not\cong}
\newcommand{\nin}{\not\in}
\newcommand{\soplus}[1]{\stackrel{#1}{\oplus}}
\newcommand{\by}[1]{\stackrel{#1}{\rightarrow}}
\newcommand{\longby}[1]{\stackrel{#1}{\longrightarrow}}
\newcommand{\vlongby}[1]{\stackrel{#1}{\mbox{\large{$\longrightarrow$}}}}
\newcommand{\ldownarrow}{\mbox{\Large{\Large{$\downarrow$}}}}
\newcommand{\lsearrow}{\mbox{\Large{$\searrow$}}}
\renewcommand{\d}{\stackrel{\mbox{\scriptsize{$\bullet$}}}{}}
\newcommand{\dlog}{{\rm dlog}\,}    
\newcommand{\longto}{\longrightarrow}
\newcommand{\vlongto}{\mbox{{\Large{$\longto$}}}}
\newcommand{\limdir}[1]{{\displaystyle{\mathop{\rm lim}_{\buildrel\longrightarrow\over{#1}}}}\,}
\newcommand{\liminv}[1]{{\displaystyle{\mathop{\rm lim}_{\buildrel\longleftarrow\over{#1}}}}\,}
\newcommand{\norm}[1]{\mbox{$\parallel{#1}\parallel$}}
\newcommand{\boxtensor}{{\Box\kern-9.03pt\raise1.42pt\hbox{$\times$}}}
\newcommand{\into}{\hookrightarrow}
\newcommand{\image}{{\rm image}\,}
\newcommand{\Lie}{{\rm Lie}\,}      
\newcommand{\CM}{\rm CM}
\newcommand{\sext}{\mbox{${\mathcal E}xt\,$}}  
\newcommand{\shom}{\mbox{${\mathcal H}om\,$}}  
\newcommand{\coker}{{\rm coker}\,}  
\newcommand{\sm}{{\rm sm}}
\newcommand{\pgcd}{\text{pgcd}}
\newcommand{\trd}{\text{tr.d.}}
\newcommand{\tensor}{\otimes}
\newcommand{\hotimes}{\hat{\otimes}}
\newcommand{\prop}{{\rm prop}}
\newcommand{\CH}{{\rm CH}}
\newcommand{\tr}{{\rm tr}}

\renewcommand{\iff}{\mbox{ $\Longleftrightarrow$ }}
\newcommand{\supp}{{\rm supp}\,}
\newcommand{\ext}[1]{\stackrel{#1}{\wedge}}
\newcommand{\onto}{\mbox{$\,\>>>\hspace{-.5cm}\to\hspace{.15cm}$}}
\newcommand{\propsubset}
{\mbox{$\textstyle{
\subseteq_{\kern-5pt\raise-1pt\hbox{\mbox{\tiny{$/$}}}}}$}}
\newcommand{\sA}{{\mathcal A}}
\newcommand{\sB}{{\mathcal B}}
\newcommand{\sC}{{\mathcal C}}
\newcommand{\sD}{{\mathcal D}}
\newcommand{\sE}{{\mathcal E}}
\newcommand{\sF}{{\mathcal F}}
\newcommand{\sG}{{\mathcal G}}
\newcommand{\sH}{{\mathcal H}}
\newcommand{\sI}{{\mathcal I}}
\newcommand{\sJ}{{\mathcal J}}
\newcommand{\sK}{{\mathcal K}}
\newcommand{\sL}{{\mathcal L}}
\newcommand{\sM}{{\mathcal M}}
\newcommand{\sN}{{\mathcal N}}
\newcommand{\sO}{{\mathcal O}}
\newcommand{\sP}{{\mathcal P}}
\newcommand{\sQ}{{\mathcal Q}}
\newcommand{\sR}{{\mathcal R}}
\newcommand{\sS}{{\mathcal S}}
\newcommand{\sT}{{\mathcal T}}
\newcommand{\sU}{{\mathcal U}}
\newcommand{\sV}{{\mathcal V}}
\newcommand{\sW}{{\mathcal W}}
\newcommand{\sX}{{\mathcal X}}
\newcommand{\sY}{{\mathcal Y}}
\newcommand{\sZ}{{\mathcal Z}}
\newcommand{\A}{{\mathbb A}}
\newcommand{\B}{{\mathbb B}}
\newcommand{\C}{{\mathbb C}}
\newcommand{\D}{{\mathbb D}}
\newcommand{\E}{{\mathbb E}}
\newcommand{\F}{{\mathbb F}}
\newcommand{\G}{{\mathbb G}}
\newcommand{\HH}{{\mathbb H}}
\newcommand{\I}{{\mathbb I}}
\newcommand{\J}{{\mathbb J}}
\newcommand{\M}{{\mathbb M}}
\newcommand{\N}{{\mathbb N}}
\renewcommand{\P}{{\mathbb P}}
\newcommand{\Q}{{\mathbb Q}}
\newcommand{\R}{{\mathbb R}}
\newcommand{\T}{{\mathbb T}}
\newcommand{\U}{{\mathbb U}}
\newcommand{\V}{{\mathbb V}}
\newcommand{\W}{{\mathbb W}}
\newcommand{\X}{{\mathbb X}}
\newcommand{\Y}{{\mathbb Y}}
\newcommand{\Z}{{\mathbb Z}}
\newcommand{\bk}{{\mathbf{k}}}
\newcommand{\bbk}{{\overline{\mathbf{k}}}}
\newcommand{\Fix}{\mathrm{Fix}}

\newcommand{\tor}{{\mathrm{tor}}}
\renewcommand{\div}{{\mathrm{div}}}

\newcommand{\trdeg}{{\mathrm{trdeg}}}
\newcommand{\Stab}{{\mathrm{Stab}}}

\newcommand{\OK}{{\overline{K}}}
\newcommand{\ok}{{\overline{k}}}

\title[]{Geometric Bogomolov conjecture in arbitrary characteristics}

\author{Junyi Xie}


\address{Univ Rennes, CNRS, IRMAR - UMR 6625, F-35000 Rennes, France}

\email{junyi.xie@univ-rennes1.fr}

\author{Xinyi Yuan}

\address{BICMR, Peking University, Haidian District, Beijing 100871, China}

\email{yxy@bicmr.pku.edu.cn}


\date{\today}

\bibliographystyle{alpha}

\maketitle

\begin{abstract}
We give a proof of the full geometric Bogomolov conjecture.
\end{abstract}

\tableofcontents

\section{Introduction}
The goal of this paper is to prove the full geometric Bogomolov conjecture.
We first reduce it to the case that the extension of the base fields has transcendence degree 1, and then we prove the later case by intersection theory in algebraic geometry. 
The proof uses Yamaki's reduction theorem on the geometric Bogomolov conjecture and the Manin--Mumford conjecture proved by Raynaud and Hrushovski.  

\subsection{Abelian varieties and heights}\label{par:1.1.1}
 
Let $k$ be an algebraically closed field.
Let $K/k$ be a finitely generated field extension of transcendence degree $\trdeg(K/k)\geq 1$.
Let $A$ be an abelian variety over $K$ of dimension $g$. 
Let $L$ be a symmetric and ample line bundle over $A$.
To define canonical height of subvarieties of $A$, we need choose integral models. 

\subsubsection{Integral models} \label{sectioncanheight}

A \emph{projective model of $K/k$} is a normal projective variety $S$ over $k$ with function field $K$.
It follows that $\dim S=\trdeg(K/k)$.

A \emph{polarization of $K/k$} is a pair $(S, \sM)$ consisting of a projective model $S$ of $K/k$ and an ample line bundle $\sM$ over $S$. 

An \emph{integral model of $(A,L)$} over $S$ is a pair $(\sA, \pi, \sL)$ where: 
\begin{points}
\item[$\bullet$] $\sA$ is a projective variety over $k$;
\item[$\bullet$] $\pi: \sA\to S$ is a projective morphism whose generic fiber is $A$;
\item[$\bullet$] $\sL$ is a line bundle on $\sA$ extending $L$.
\end{points}
We usually abbreviate $(\sA, \pi, \sL)$ as $(\sA, \sL)$.

\subsubsection{Canonical heights}

Let $X$ be any closed subvariety of $A$ over $K$.
The {\bf{naive height}} of $X$ with respect to the polarization $(S,\sM)$
and the integral model $(\sA, \sL)$ is defined as 
\begin{equation}
\label{eq:naiveheight}
h_{(\sA,\sL)}^{\sM}(X):=\frac{\sL^{\dim X+1}\cdot (\pi^*\sM)^{\dim S-1}\cdot \sX}{(1+\dim X)\deg_L(X)},
\end{equation}
where $\sX$ is the Zariski closure of $X$ in $\sA$, the numerator is the intersection number in $\sA$, and $\deg_L(X)=L^{\dim X}\cdot X$
is the intersection number in $A$.

By Tate's limiting argument, the {\bf{canonical height}} of $X$ with respect to the polarization $(S,\sM)$
and the line bundle $L$ over $A$ is the limit
\begin{equation}\label{eq:ntlimit}
\hat{h}_L^{\sM}(X)
:=\lim_{n\to +\infty} \frac{1}{n^2}h_{(\sA,\sL)}([n]X).
\end{equation}
Here $[n]X$ is the image of $X$ under the multiplication map $[n]:A\to A$.
The limit exists, depends on $(S,\sM)$ and $L$, but is independent of the integral model $(\sA, m, \sL)$; see Gubler's work \cite[Theorem 3.6]{Gubler2007} and \cite[Theorem 11.18]{Gubler:Pisa03}.
Moreover, $\hat{h}_L^{\sM}(X)\geq0$ by choosing $\sL$ to be ample in the definition.
Note that our heights are normalized by the degree and dimension in the denominator. 

The definition of the heights can be easily extended to subvarieties of $A_{\overline{K}}=A\otimes_K\overline{K}$. In fact, let $X'$ be a closed subvariety of $A_{\overline{K}}$.
Let $X$ be the minimal closed subvariety of $A$ such that $X_{\overline K}$ contains $X'$ in $A_{\overline{K}}$, or equivalently $X$ is the image of the composition 
$X'\to A_{\overline{K}}\to A$ as schemes. 
Then we simply define
\begin{equation}\label{eq:ntlimit}
h_{(\sA,m,\sL)}^{\sM}(X')
:=h_{(\sA,m,\sL)}^{\sM}(X),\qquad
\hat{h}_L^{\sM}(X')
:=\hat{h}_L^{\sM}(X).
\end{equation}
For simplicity, we usually omit the dependence on $L$ and $(S,\sM)$ and write 
$$\hat{h}(X)=\hat{h}_L^{\sM}(X),\qquad
\hat{h}(X')=\hat{h}_L^{\sM}(X').$$

In the special case of $\dim X'=0$, we get the canonical height  for points 
$$\hat{h}=\hat{h}_L^{\sM}:A(\overline{K})\to [0,+\infty)$$ 
with respect to $L$ and $(S,\sM)$.

In the special case of $\trdeg(K/k)=1$, the situation is much easier, since the the projective model $S$ is unique and the line bundle $\sM$ is not used in the definition of the heights.

\subsubsection{Small points and special subvarieties}
For any subvariety $X$ of  $A_{\overline{K}}$ and any $\epsilon> 0$, we set 
\begin{equation}
X(\epsilon):=\{x\in X(\overline{K})|\,\, \hat{h}(x)<\epsilon\}.
\end{equation}
We say that $X$ \emph{contains a dense set of small points of $A/K/k$} if $X(\epsilon)$ is Zariski dense in $X$ for all $\epsilon >0$.
This notion is actually independent of the choice of $L$ and $(S,\sM)$ to define the canonical height $\hat h$, since any two such heights bound each other up to positive  multiples.

Let $(A^{\overline{K}/k}, \tr)$ be the $\overline{K}/k$-trace of  $A_{\overline{K}}$; it
is the final object of the category of pairs $(C,f)$, where $C$ is an abelian variety over $ k$ and $f$ is 
a morphism from $C\otimes_{ k}\overline{K}$ to  $A_{\overline{K}}$ (cf. \cite[$\mathsection$7]{Lang:AbelianBook} or~\cite[$\mathsection$6]{Conrad}).
If $\Char k=0$, $\tr$ is a closed immersion and $A^{\overline{K}/ k}\otimes_{k}\overline{K}$ can be naturally viewed 
as an abelian subvariety of $A_{\overline{K}}$;
if $\Char k>0$, $\tr$ is a purely inseparable isogeny to its image. 

A {\bf{torsion subvariety}} of $A_{\overline K}$ is a translate $a+C$ of an abelian 
subvariety $C\subset A_{\overline K}$ by a torsion point $a$ of $A_{\overline K}$. A subvariety $X$ of $A_{\overline{K}}$ is said to be {\bf{special}} if  
\begin{equation}
X=\tr(Y{\otimes_{k}\overline{K}})+T
\end{equation}
for some torsion subvariety $T$ of $A_{\overline{K}}$ and  some subvariety $Y$ of $A^{\overline{K}/k}$.
When $X$ is special,  $X(\epsilon)$ is Zariski dense in $X$ for all $\epsilon >0$ (\cite[Theorem 5.4, Chapter 6]{Lang1983}).

\subsection{Geometric Bogomolov conjecture}

The goal of this paper is to prove the following theorem, which is known as the geometric Bogomolov conjecture. 

\begin{thm}[geometric Bogomolov conjecture] \label{thmmain} 
Let $k$ be an algebraically closed field.
Let $K/k$ be a finitely generated field extension of transcendence degree at least 1.
Let $A$ be an abelian variety over $K$. 
Let $X$ be a closed subvariety of $A_{\overline{K}}$. 
If $X$ contains a dense set of small points of $A/K/k$, then $X$ is special.
\end{thm}

The above geometric Bogomolov conjecture was proposed by Yamaki \cite[Conjecture 0.3]{Yamaki2013}, but particular
 instances of it were studied earlier by Gubler in~\cite{Gubler2007}. It is an analog over function fields of the original Bogomolov conjecture over number fields which was proved by Ullmo \cite{Ullmo1998} and Zhang \cite{Zhang1998}.

Let us give a quick historical recall on results about the Bogomolov conjecture and its geometric version.
The original Bogomolov conjecture over number fields was proved by Ullmo \cite{Ullmo1998} and Zhang \cite{Zhang1998}, where the major techniques are the equidistribution theorem of Szpiro--Ullmo--Zhang \cite{SUZ}.
The treatments are extended in terms of the Moriwaki height to finitely generated fields over number fields by Moriwaki \cite{Moriwaki2000}.

For the geometric Bogomolov conjecture, it was first proved by Gubler \cite{Gubler2007} when $A$ is totally degenerate at some place of $K/k$.
Yamaki \cite{Yamaki2016c, Yamaki2016a} reduced the geometric Bogomolov conjecture to the case of abelian varieties with good reduction everywhere and with a trivial trace; based on this reduction theorem, Yamaki \cite{Yamaki2017} proved the conjecture for $\dim(X)=1$ or $\mathrm{codim}(X)=1$; he also proved the conjecture for abelian varieties of dimension 5, good reduction everywhere and with trivial trace in \cite{Yamaki2017a}.

The works of Gubler \cite{Gubler2007} and Yamaki \cite{Yamaki2016c} are 
valid in arbitrary characteristics, and extend the strategy of Ullmo and Zhang applying equidistribution theorems. 
In fact, Gubler \cite{Gubler2007} considered tropical varieties of subvarieties of abelian varieties over non-archimedean fields, 
extended the equidistribution theorem of Szpiro--Ullmo--Zhang to the tropical setting, and studied the equilibrium measure in this setting.
Yamaki \cite{Yamaki2016c} did more careful analysis of the situation using equidistribution over Berkovich spaces. The latter equidistribution theorem was proved by Gubler \cite{Guber2008} and Faber \cite{Faber2009b}, which generalized the equidistribution theorems over number fields of Szpiro--Ullmo--Zhang \cite{SUZ},  Chambert-Loir \cite{CL2006}, and Yuan \cite{Yuan2008}.

Before Yamaki \cite{Yamaki2017}, various examples and partial results of the geometric Bogomolov conjecture with $\dim(X)=1$ were previously obtained by \cite{Parshin, Moriwaki1996, Moriwaki1997, Moriwaki1998, Yamaki2002, Yamaki2008, Faber2009, Cinkir2011}.
In particular, Cinkir \cite{Cinkir2011} proved the geometric Bogomolov conjecture for 
$\dim(X)=1$ and $\Char k=0$, based on a height identity of Zhang \cite{Zhang2010}.

In the case $\Char k=0$, Gao--Habegger  \cite{Gao2019} proved the geometric Bogomolov conjecture for $\trdeg(K/k)=1$. 
Recently Cantat--Gao--Habegger--Xie \cite{Cantat2021} proved the full geometric Bogomolov conjecture for $\Char k=0$. 
These proofs are based on the Betti map in the complex-analytic setting.

\subsection{Plan of proof}

Our proof of the geometric Bogomolov conjecture is based on Yamaki's reduction theorem, which reduces the conjecture to the case of good reduction and trivial trace.
We also need the Manin--Mumford conjecture (in the case of trivial trace) proved by Raynaud and Hrushovski.

First, we reduce the conjecture to the case that $K/k$ has transcendence degree 1. This is the main goal of \S\ref{sec trdeg}.
The idea is to take intermediate fields $k'$ of $K/k$, algebraically closed in $K$ and with transcendence degree 1 over $k$, such that the geometric Bogomolov conjecture for $(A/K/k,X)$ follows from those of $(A/K/k',X)$.
For the construction, let $(S,\sM)$ be a polarization of $K/k$.
Use $\sM$ to get a pencil of hyperplane sections of $S$ 
parametrized by $\P_k^1$. Take $k'=k(\P_k^1)$ to be the function field.
There is a generic hyperplane section $H$ over the generic point $\Spec k'$ of $\P_k^1$. 
In particular, $(H,\sM|_H)$ is a polarization of $K/k'$.  
The process from $(S,\sM)$ to $(H,\sM|_H)$ does not increase canonical heights of subvarieties of $A_\OK$. 
Then we carry out careful analysis of the change of special subvarieties of $A$ under this process.

We remark that a well-known procedure to reduce the transcendence degree is to take a closed hyperplane section of 
of $S$ over $k$. This process reduces $K/k$ to $K'/k$ (instead of our $K/k'$), and thus changes $X$ and $A$ by the corresponding reductions. Because of this, it is hard to track density of small points, and it is also hard to track the change of the trace of $A$.
Our method does not change $K$ and thus make everything trackable.

As a consequence, we can assume that $K/k$ has transcendence degree 1. 
Then we apply Yamaki's reduction theorem. 
This reduces the problem to the essential case that $K/k$ has transcendence degree 1, 
and $A$ has good reduction over $S$ and trivial $K/k$-trace. 
Here $S$ is the unique smooth projective curve over $k$ with function field $K$.
Then $A\to \Spec K$ extends to an abelian scheme $\pi:\sA\to S$.

The second step is to prove the Bogomolov conjecture in the essential case. 
Denote by $\sX$ the Zariski closure of $X$ in $A$.
Let $\sL$ be a symmetric, relatively ample and rigidified line bundle over $\sA$.
Our key property is that any torsion multi-section $\sT$ of $\sA\to S$ is numerically equivalent to a multiple of the self-intersection $\sL^{\dim A}$ in the Chow group of 1-cycles in $\sA$.
This is proved in \S\ref{sec line bundle}.

Now we come to \S\ref{sec final}, which is the core of the proof of the essential case.
To illustrate the idea, assume that there is a positive integer $r$ such that the summation map $f:X^r\to A$ is surjective and generically finite. This happens, for example, if $X$ is a curve and $A$ is the Jacobian variety of $X$.
We have a morphism $f:\sX_{/S}^r\to \sA$ over $S$. 
Note that $X^r\subset A^r$ and $\sX_{/S}^r\subset \sA_{/S}^r$.
Moreover, $X^r$ has canonical height 0 in $A^r$, by Zhang's fundamental inequality and the assumption that $X$ has a dense set of small points. 

Let $\sT$ be a torsion multi-section of $\sA$. 
Let us first assume that $\sT'=f^{-1}(\sT)$ is finite and flat over $S$, and address the technical issue later. 
Denote by $h:\sA_{/S}^r\to \sA$ the summation morphism. 
For any symmetric, relatively ample and rigidified line bundle 
$\sL_r$ over $\sA_{/S}^r$, we have
$$
[\sT']\cdot \sL_r^{\dim \sT'}=[\sX_{/S}^r] \cdot h^*[\sT]\cdot \sL_r^{\dim \sT'}
=a\, [\sX_{/S}^r] \cdot h^*(\sL^{\dim A})\cdot \sL_r^{\dim \sT'}
=0.
$$
Here $a>0$ is a constant coming from the numerical equivalence mentioned above. 
The last equality follows from the fact that $X^r$ has canonical height 0 in $A^r$.
This implies that $T'=\sT'_K$ has canonical height 0, and thus is a finite set of torsion points of $A^r$. 
When varying $\sT$, we obtain a Zariski dense set of torsion points of $X^r$ in $A^r$. 
By the Manin-Mumford conjecture, $X^r$ is torsion, and thus $X$ is torsion. 

In general, we are not able to find $r$ such that $f:X^r\to A$ is surjective and generically finite.
But we can manage to find $r$ such that $f$ is surjective (up to replacing $A$ by an abelian subvariety), and that the relative dimension $e$ of $f$ is strictly smaller than $\dim X$.
If $\sT'=f^{-1}(\sT)$ is flat over $S$ of dimension $e+1$, then the above process still implies that $T'=\sT'_K$ has canonical height 0. 
Then we can conclude that $T'$ is a torsion subvariety in $A^r$ by induction, since $\dim T'<\dim X$.
Torsion points of $T'$ are also torsion points of $X^r$. 
Varying $T'$, we obtain a dense set of torsion points of $X^r$ in $A^r$. 
Then the Manin-Mumford conjecture implies that $X$ is torsion.

In the above, we have made a technical assumption that $\sT'=f^{-1}(\sT)$ is flat over $S$ of dimension $e+1$.
To remove this caveat, we will have to treat the case that $\sT'=f^{-1}(\sT)$ has irreducible components of dimension bigger than the expected dimension. Note that it is easy to choose $\sT$ so that the Zariski closure $\sT^*$ of $T'=\sT'_K$ in $\sT'$ has the correct dimension. 
Then we prove that, the difference $[\sX_{/S}^r] \cdot h^*[\sT]-[\sT^*]$ is linearly equivalent to an effective Chow cycle of the correct dimension. 
It remedies the above argument. 
This is the content of \S\ref{sec nonproper}, and was inspired by a result of Jia--Shibata--Xie--Zhang.

The idea of converting the Bogomolov conjecture to the Manin--Mumford conjecture was originally used by Zhang \cite{Zhang1992, Zhang1995a} in his proof of the Bogomolov conjecture for powers of $\G_m$. 
In particular, \cite[Lemma 6.6]{Zhang1995a} produced enough torsion hypersurfaces of $X$, and thus enough torsion points of $X$ by induction.  
On the other hand, our situation is more complicated due to lack of torsion hypersufaces in $A$, and our solution is based on the crucial numerical identity between $\sT$ and $\sL^{\dim A}$. As these cycles are not of codimension 1, it also brings issues of non-proper intersection discussed above.

\subsection{Notation and Terminology}
\begin{points}
\item[$\bullet$] For any field $F$, denote by $\overline F$ its algebraic closure.
\item[$\bullet$] For a field extension $K/k$, denote by $\trdeg(K/k)$ the transcendence degree.
\item[$\bullet$] A \emph{variety} is an integral separated scheme of finite type over a field.
\item[$\bullet$] For a Cartier divisor $H$ on a scheme $X$, denote by $|H|$ the linear system associated to $H.$
\item[$\bullet$] For an integral scheme $X$, denote by $\eta_X$ its generic point.
\item[$\bullet$] For a scheme $X$ over a field, denote by $X^{\sm}$ its smooth locus.
\item[$\bullet$] A \emph{subvariety} of a variety is a closed integral subscheme.
\item[$\bullet$]
By a \emph{line bundle} over a scheme, we mean an invertible sheaf over the scheme. 
We often write or mention tensor products of line bundles additively, so $aL-bM$ means
$L^{\otimes a}\otimes M^{\otimes (-b)}$
for line bundles $L,M$ and integers $a,b$.
\end{points}

\medskip

\noindent\textbf{Acknowlegement.}
The authors would like to thank the support of the China-Russia Mathematics Center during the preparation of this paper.
The first-named author would like to thank Beijing International Center for Mathematical Research in Peking University 
for the invitation. The second-named author would also like to thank the hospitality of Shandong University, Qingdao for a workshop in July 2021.

\section{Non-proper intersections} \label{sec nonproper}

Let $B$ be a smooth projective variety of dimension $d$ over an algebraically closed field $k$.
For every $i$-cycle $Z$ of $\Q$-coefficient of $B$,  denote by $[Z]$ its class in the Chow group
$\CH_i(B)_{\Q}$. For $\alpha\in \CH_i(B)_{\Q}$, write $\alpha\geq 0$ if $\alpha$ can be represented by an effective $i$-cycle.
The goal of this section is to prove the following technical result. 

\begin{pro}\label{prostrleqtot}
Let $g: B\to Y$ be a flat morphism between smooth projective varieties.
Let $X$ be a closed subvariety of $B$ such that $f=g|_X: X\to Y$ is surjective. 
Denote $e=\dim X-\dim Y$. 
Let $V$ be a closed subvariety of $Y$ which is not contained in 
$$Y_{e+1}=\{y\in Y|\,\, \dim f^{-1}(y)\geq e+1\}.$$ 
Let $Z_1, \dots, Z_n$ be all irreducible components of $f^{-1}(V)$ satisfying 
$f(Z_i)=V.$ Then $\dim Z_i=\dim V+e$ for $i=1,\dots,n$. 

Assume furthermore that $V\cap Y_{e+1}$ is finite and contained in $V^{\sm}$.
Then we have 
$$\sum_{i=1}^n m_i[Z_i]\leq g^*[V]\cdot [X]$$
in $\CH_{\dim V+e}(B)_\Q$. 
Here $m_i$ is the multiplicity of $Z_i$ in $f^{-1}(V).$ 
\end{pro}

Note that if $X$ is smooth, then the last inequality can be simplified as the inequality 
$$\sum_{i=1}^n m_iZ_i\leq f^*[V]$$
in $\CH_{\dim V+e}(X)_\Q$. 

The proposition will be used in \S\ref{sec final}. 
The remaining part of this section is to prove the proposition. 
Readers might skip the proof temporarily and move to the next section at the first time of reading this paper.

The idea to prove the proposition is to use complete intersections of hyperplane sections to bound the proper part of the intersection. See Proposition \ref{propropintboundhyp} for the result on complete intersection. 
In the following, we will start with a Bertini type of result to choose suitable hyperplane sections.

\subsection{A Bertini type of result}

\begin{pro}\label{probertini}Let $Y$ be a projective variety over an algebraically closed field $k$.
Let $V$ be a closed subvariety of $Y$ of codimension $r\geq 1$ such that $Y$ is smooth at $\eta_V.$
Let $Z_1,\dots,Z_m$ be irreducible subvarieties of $Y$.
Assume that 
\begin{points}
\item[$\bullet$] $(\cup_{i=1}^m Z_i)\cap V$ is finite; 
\item[$\bullet$] $(\cup_{i=1}^m Z_i)\cap V\subseteq Y^{\sm}\cap V^{\sm}$.
\end{points}
Let $H$ be a Cartier divisor on $Y$ such that $\sO_Y(H)\otimes I_V$ is generated by global sections, where
 $I_V\subset \sO_Y$ is the ideal sheaf associated to $V$.
Every section $s\in H^0(\sO_Y(H)\otimes I_V)$ defines a divisor 
$H_s\in |H|$ via the inclusion $H^0(\sO_Y(H)\otimes I_V)\hookrightarrow H^0(\sO_Y(H)).$
Then for general elements 
$$(s_1,\dots, s_r)\in H^0(\sO_Y(H)\otimes I_V)^r,$$ 
the following holds:
\begin{points}
\item $H_{s_1}\cap\dots\cap H_{s_r}$ is a proper intersection in $Y$;
\item $H_{s_1}\cdots H_{s_r}=V+W$ where $W$ is an effective $(\dim Y-r)$-cycle such that 
$V\not\subseteq \Supp \,W$, and $W\cap Z_i$ is a proper intersection for every $i=1,\dots,m$.
\end{points}
\end{pro}

\proof
Set $H_i:=H_{s_i}$ for $i=1,\dots,r.$
Because $\sO_Y(H)\otimes I_V$ is generated by global sections and $s_1,\dots, s_r$ are general in $H^0(\sO_Y(H)\otimes I_V)$, 
$(H_1\setminus V)\cap\dots\cap  (H_r\setminus V)$ is a proper intersection in $Y\setminus V$, and 
$(H_1\setminus V)\cap\dots\cap  (H_r\setminus V)\cap (Z_i\setminus V)$ is proper intersection in $Y\setminus V$ for every $i=1,\dots,m$.

\medskip

As $V\subseteq H_i,$ the intersection 
$$H_1\cap \dots \cap H_r= V\cup ((H_1\setminus V)\cap\dots\cap  (H_r\setminus V))
$$
 is of pure codimension $r$ in $Y$.
So $H_1\cap \dots \cap H_r$ is a proper intersection. 
Write $H_{1}\cdots H_{r}=V+W$ where $W$ is an effective $(\dim Y-r)$-cycle.

\medskip
Pick $x_V\in Y^{\sm}(k)\cap V^{\sm}(k)$ and set $B:=(V\cap (\cup_{i=1}^m Z_i))\cup \{x_V\}$. It is a finite subset of $Y^{\sm}(k)\cap V^{\sm}(k).$
For every point $x\in B$, fix an isomorphism $\phi_x: (O(H)\otimes I_V)_x\to I_{V,x}.$  Because $x\in (V^{\sm}\cap Y^{\sm})(k)$, $I_{V,x}/ I_{V,x}I_x$ is a $k$-vector space of dimension $r.$ Because $\sO_Y(H)\otimes I_V$ is generated by global sections and $s_1, \dots, s_r$ are general in $H^0(\sO_Y(H)\otimes I_V)$, $\phi_x(s_1), \dots, \phi_x(s_r)$ is a $k$-basis of 
$I_{V,x}/ I_{V,x}I_x$. Then $\phi_x(s_1),\dots, \phi_x(s_r)$ generate $I_{V,x}.$ This implies that at every point $x\in B$, there is an open neighborhood $U_x$ of $x$ such that $H_1\cap \dots \cap H_r\cap U_x=V\cap U_x.$ 
So $V\not\subseteq \Supp \,W$, and $W\cap B=\emptyset.$
Because $V\cap (\cup_{i=1}^m Z_i)\subseteq B$ for every $i=1,\dots, m$, 
the intersection
$$W\cap Z_i=W\cap (Z_i\setminus V)=(H_1\setminus V)\cap\dots\cap  (H_r\setminus V)\cap (Z_i\setminus V)$$ is a proper intersection.
\endproof

\subsection{Proper part of an intersection}

\begin{lem}\label{lemcutampleeff}
Let $B$ be a smooth projective variety of dimension $d$ over an algebraically closed field $k$.
Let $H$ be a Cartier divisor on $B$ such that $\sO_B(H)$ is globally generated. 
Then for $\alpha\in \CH_i(B)_{\Q}$ with $\alpha\geq 0$, $\alpha\cdot [H] \geq 0.$
\end{lem}

\proof
We may write $\alpha=[Z]$ for an effective $i$-cycle $Z$. We can further assume that $Z$ is an integral subvariety of $B$. 
Because $\sO_B(H)$ is globally generated, after replacing $H$ by a general element in $|H|$, we may assume that $H\cap Z$ is a proper intersection.
Then $\alpha\cdot [H]=[Z]\cdot [H]\geq 0.$
\endproof

\medskip 

Let $B$ be a smooth projective variety of dimension $d$ over an algebraically closed field $k$.
Let $H_1,\dots,H_m$ be Cartier divisors of $B$. Let $X$ be a subvariety of $B.$
A \emph{proper component} $Z$ of $H_1\cap \dots \cap H_m\cap X$ is an irreducible component of the underling topological space of $H_1\cap \dots \cap H_m\cap X$ of dimension $\dim X-m.$ Write $m(Z, H_1\cap \dots \cap H_m\cap X)$ for the multiplicity of $Z$ in $H_1\cap \dots \cap H_m\cap X$, as defined by Serre using the tor-functor.
Define 
$$(H_1\cdots H_m\cdot X)^{\prop}:=\sum m(Z, H_1\cap \dots \cap H_m\cap X) [Z],$$
where the  sum takes over all proper components $Z$ of $H_1\cap \dots \cap H_m\cap X$.

\medskip

The following result is a generalization of \cite[Lemma 3.3]{Jia2021} with a similar proof.

\begin{pro}\label{propropintboundhyp}
Let $H_1, \dots, H_r$ be effective Cartier divisors on $B$ such that 
$\sO_B(H_1),\dots,\sO_B(H_r)$ are generated by global sections. 
Let $X$ be a subvariety of $B$.
Then we have 
$$(H_1\cdots H_r\cdot X)^{\prop}\leq H_1\cdots H_r\cdot X.$$
\end{pro}

\proof 
Let $V_1,\dots, V_m$ be the proper components of $H_1\cap \dots \cap H_r\cap X$ with multiplicities $m_1,\dots,m_r.$
For $i=1,\dots,m,$ set $\eta_i:=\eta_{V_i}$. Then $H_1\cap \dots \cap H_r\cap X$ has proper intersection at $\eta_1,\dots, \eta_m.$

Let $X_1,\dots, X_l$ be all irreducible components of $H_1\cap\dots\cap H_{r-1}\cap X$ passing through $\eta_1,\dots, \eta_m.$
For  $j=1,\dots,l$, the variety $X_j$ has dimension $\dim X-r+1$ and $X_j\not\subseteq H_r$
Assume that $X_j$ has multiplicity $n_j$ in $H_1\cap\dots\cap H_{r-1}\cap X$.

If $r=1,$ this lemma is trivial.  Now assume that $r\geq 2$. 
By induction,
 $$\sum_{j=1}^l n_jX_j\leq H_1\cdots H_{r-1}\cdot X.$$
The previous paragraph shows that $(\sum_{j=1}^l n_jX_j)\cap H_r$ is a proper intersection and $m_i$ is the multiplicity of $V_i$ in $(\sum_{i=1}^l n_jX_j)\cap H_r.$
By Lemma \ref{lemcutampleeff}, we have 
$$(H_1\cdots H_r\cdot X)^{\prop}=\sum_{i=1}^m m_iV_i\leq (\sum_{j=1}^l n_jX_j)\cap H_r\leq H_1\cdots H_{r-1}\cdot V\cdot H_r.$$
\endproof

\subsection{Strict transform}
Let $f:X\to Y$ be a surjective morphism of projective varieties over an algebraically closed field $k.$ 
Set $e:=\dim X-\dim Y$.

For every integer $l\geq e$,   
$$Y_l:=\{y\in Y|\,\, \dim f^{-1}(y)\geq l\}$$ 
is a closed subset of $Y$. 
We have $Y_{l+1}\subseteq Y_l$.
We note that $Y_{e}=Y$ and for $l\geq e+1,$
$\dim Y_l+l\leq \dim X-1$. In particular, for $l\geq \dim X$, $Y_l=\emptyset.$

\begin{lem}\label{lempullbackproper}
Let $W$ be a subvariety of $Y$. 
Assume that $W\cap Y_l$ is a proper intersection for every $l\geq e+1$.
Then 
$$\dim f^{-1}(W)=\dim W+e.$$
Moreover, for every irreducible component $Z$ of $f^{-1}(W)$ with $\dim Z=\dim f^{-1}(W)$, we have $f(Z)=W.$
\end{lem}

\proof
Write $W=\sqcup_{e\leq l\leq \dim X} W\cap(Y_l\setminus Y_{l+1}).$
For every $l\geq e+1$, if $W\cap(Y_l\setminus Y_{l+1})\neq \emptyset,$
$$\dim W\cap(Y_l\setminus Y_{l+1})\leq \dim W\cap Y_l=\dim W+\dim Y_l-\dim Y$$
$$\leq \dim W+\dim X-l-1-\dim Y=\dim W+e-l-1.$$
So for $l\geq e+1$, if $W\cap(Y_l\setminus Y_{l+1})\neq \emptyset,$ 
$$\dim f^{-1}(W\cap (Y_l\setminus Y_{l+1}))\leq \dim W+e-1.$$

\medskip

Because $W\cap Y_{e+1}$ is a proper intersection, we have $W\setminus Y_{e+1}\neq \emptyset.$
Then $$\dim f^{-1}(W\setminus Y_{e+1})=\dim W+e.$$
So we have $$\dim f^{-1}(W)=\dim W+e.$$
Let $Z$ be an irreducible component of $f^{-1}(W)$ of $\dim Z=\dim f^{-1}(W)=\dim W+e$.
Then $\dim f(Z)\geq \dim Z-e= \dim W.$ So $f(Z)=W.$
\endproof

\proof[Proof of Proposition \ref{prostrleqtot}]
For every $i=1,\dots,n$, since $f(Z_i)=V$ and $Z_i$ is irreducible, 
$\dim Z_i=\dim (Z_i\setminus f^{-1}(Y_{e+1})).$
Since $f^{-1}(V\setminus Y_{e+1})$ is of pure dimension $e+\dim V$,  $\dim Z_i=e+\dim V.$ 

Set $r:=\dim Y-\dim V.$ 
By Proposition \ref{probertini},  there are effective and very ample divisors $H_1,\dots, H_r$ on $Y$ such that 
$H_1\cap \dots \cap H_r$ is a proper intersection and 
$$H_1\cdots H_r=V+W$$
where $W$ is an effective $(\dim Y-r)$-cycle such that 
$V\not\subseteq \Supp \,W$, and such that $W\cap Y_l$ is a proper intersection
for $l\geq e+1$.
We have $$g^*[V]+g^*[W]=[g^*H_1\cdots g^*H_r].$$

By Lemma \ref{lempullbackproper}, $\dim f^{-1}W=\dim V+e.$  
Then $f^{-1}W=g^{-1}W \cap X$ is a proper intersection. Since $B$ is smooth, $f^{-1}W$ is equidimensional. 
By Lemma \ref{lempullbackproper} again, for every irreducible component $R$ of $f^{-1}W$, the image $f(R)$ is an irreducible component of $W.$

We claim that $[f^{-1}W]=g^*[W]\cdot [X]$ as algebraic cycles over $B$. 
In fact, it suffices to prove the restriction of the equality to $U=B\setminus g^{-1}(V)$. Over $Y\setminus V$, $W$ is the complete intersection $H_1\cap \dots \cap H_r$. 
Then the result follows by the fact that the intersection multiplicities are given by length of the local rings. 
This can be obtained by the vanishing of the higher tor-functors in Serre's intersection formula, or as an example of \cite[Proposition 7.1 (b)]{Fulton1984}. 

Similarly, we have the Zariski closure
$$[\overline{f^{-1}(V)\cap f^{-1}(Y\setminus Y_{e+1})}]
=\sum_{i=1}^n m_i[Z_i]$$
as algebraic cycles over $B$. 
The sum gives 
$$[\overline{f^{-1}(H_1\cap \dots \cap H_r)\cap f^{-1}(Y\setminus Y_{e+1})}]=\sum_{i=1}^n m_i[Z_i]+g^*[W]\cdot [X].$$

Note that every irreducible component of 
$f^{-1}(H_1\cap \dots \cap H_r)\cap f^{-1}(Y\setminus Y_{e+1})$ has dimension $\dim V+e.$
So we have 
$$[\overline{f^{-1}(H_1\cap \dots \cap H_r)\cap f^{-1}(Y\setminus Y_{e+1})}]\leq (g^*H_1\cdots g^*H_r\cdot X)^{\prop}$$
in $\CH_{\dim V+e}(B)_\Q$. 

Finally, by Proposition \ref{propropintboundhyp}, since $\sO_{B}(f^*H_i)$ for $ i=1,\dots,r$ are generated by global sections, 
we get 
$$(g^*H_1\cdots g^*H_r\cdot X)^{\prop}\leq g^*H_1\cdots g^*H_r\cdot [X].$$
It follows that
$$\sum_{i=1}^n m_i[Z_i]+g^*[W]\cdot [X]\leq g^*H_1\cdots g^*H_r\cdot [X]=g^*[V]\cdot X+g^*[W]\cdot [X].$$
This concludes the proof.
\endproof

\section{Lowering the transcendence degree}\label{sec trdeg}

The geometric Bogomolov conjecture concerns a finitely generated extension $K/k$.
The goal of this section is to lower $\trdeg(K/k)$ to 1 in the conjecture.
The main result of this section is as follows. 

\begin{pro}\label{changefield} 
Let $k$ be an algebraically closed field.
Let $K/k$ be a finitely generated field extension of transcendence degree at least 2.
Let $A$ be an abelian variety over $K$. 
Then there are two intermediate fields $k_1,k_2$ of $K/k$, algebraically closed in $K$ and of transcendence degree $1$ over $k$, such that for any closed subvariety $X$ of $A_{\overline{K}}$,  the geometric Bogomolov conjecture holds for $(A/K/k, X)$ if the geometric Bogomolov conjecture holds for $(A_{K\overline k_1}/K\overline k_1/\overline k_1,X)$ and $(A_{K\overline k_2}/K\overline k_2/\overline k_2,X)$.
\end{pro}

Note that the intermediate fields $k_1,k_2$ do not depend on the subvariety $X$.
Applying the theorem repeatedly, we reduce the geometric Bogomolov conjecture to the case $\trdeg(K/k)=1$.

\subsection{Field of definition}
For two abelian schemes $A_1,A_2$ over a base scheme,  denote $A_1\sim A_2$ if $A_1$ is isogenous to $A_2.$
For an abelian variety over a field $K$ and a subfield $k$ of $K$, we say $A$ is defined over $k$ up to isogeny if there is an abelian variety $A'$ over $k$ such that $A\sim A'\otimes_kK.$

\begin{pro}\label{lembasechangediagram}
Let $k$ be an algebraically closed field. Let $K$ be a field extension of $k$ with $\trdeg(K/k)<\infty.$
Let $k_1,k_2$ be algebraically closed intermediate fields of $K/k$ with $k_1\cap k_2=k.$
Let $A_1,A_2$ be abelian varieties over $k_1,k_2$ respectively. 
 If $A_1\otimes_{k_1}K\sim A_2\otimes_{k_2}K$, then there is an abelian variety
$A$ over $k$, such that $A_1\sim A\otimes_{k}k_1$ and $A_2\sim A\otimes_{k}k_2.$
Moreover, the abelian variety $A$ is unique up to isogeny. 
\end{pro}

For the uniqueness in the proposition, we have the following result. 

\begin{lem}\label{lemisobasechange}
Let $A_1,A_2$ be two abelian varieties over an algebraically closed field $k$.  Let $K$ be any field extension of $k$. If $A_1\otimes_kK\sim A_2\otimes_kK$, then $A_1\sim A_2.$
\end{lem}

\proof
There is an isogeny  $\Phi: A_1\otimes_kK\to A_2\otimes_kK$ over $K.$ There is a subfield $K'$ of $K$, finite generated over $k$, such that $\Phi$ is defined over $K'$. After replacing $K$ by $K'$, we may assume that $K$ is a finitely generated extension over $k$.
There is a $k$-variety $S$ such that $K=k(S).$ After shrinking $S$, we may assume that there is an isogeny $\Phi_S: A_1\times_k S\to A_2\times_k S$ over $S$ such that $\Phi$ is the generic fiber of $\Phi_S.$ Pick a point $b\in S(k)$. The restriction of $\Phi_S$ to the fiber at $b$ induces an isogeny $\Phi_b: A_1\to A_2$, which concludes the proof.
\endproof

\proof[Proof of Propositon \ref{lembasechangediagram}]  
We have noted that Lemma \ref{lemisobasechange} implies the uniqueness of $A.$
For the existence, we only need to show that both $A_1$ and $A_2$ are defined over $k$ up to isogeny.
Indeed, if $A_1\sim A_1'\otimes_kk_1$ and  $A_2\sim A_2'\otimes_kk_2$ for abelian varieties $A_1',A_2'$ over $k$, then $A_1'\otimes_kK\sim A_2'\otimes_kK,$
which implies $A_1'\sim A_2'$ by Lemma \ref{lemisobasechange} again.

Now we prove that $A_1$ and $A_2$ are defined over $k$ up to isogeny.
By Lemma \ref{lemisobasechange}, there is an isogeny 
$A_1\otimes_{k_1}K\sim A_2\otimes_{k_2}K$ defined over the algebraic closure of $k_1k_2$, and thus defined over a finite extension of $k_1k_2$. 
Thus we may assume that $K/k_1k_2$ is finite.

For $i=1,2$, there is a $k$-variety $S_i$, an abelian scheme $\pi:\sA\to S_i$ such that $\overline{k(S_i)}=k_i$ and $A_i\to \Spec k_i$ is the geometric generic fiber of $\pi:\sA\to S_i.$

Because $A_1\otimes_{k_1}K\sim A_2\otimes_{k_2}K$, there is $k$-variety $U$ satisfying $k(U)=K$, a flat and quasi-finite morphism $\psi: U\to S_1\times S_2$, and an isogeny
$$\Phi_U:\sA_{1,U}=\sA_1\times_{S_1} U\longrightarrow \sA_{2,U}=\sA_2\times_{S_2} U.$$
Pick a point $s=(s_1,s_2)\in \psi(U)(k)\subseteq S_1\times S_2.$ 
Consider $S_1\times s_2 \subseteq S_1\times S_2.$
Pick an irreducible component $V$  of $\psi^{-1}(S_1\times s_2)\subseteq U.$
Then $\Phi_U$ induces an isogeny 
$$\Phi_V=\Phi_U\times_UV: \sA_{1,U}\times_UV\longrightarrow \sA_{2,U}\times_UV.$$
The isomorphism $S_1\simeq S_1\times s_2$ induces an isomorphism 
\begin{equation}\label{equDlu}\sA_{1,U}\times_UV\simeq \sA_1\times _{S_1} V.\end{equation}
On the other hand, 
\begin{equation}\label{equDmu}
\sA_{2,U}\times_UV\simeq \sA_{2,s_2}\times V,\end{equation}
where $\sA_{2,s_2}:=\sA_2\times_{S_2}s_2$ is an abelian variety over $k.$
Because $k(V)/k(S_1)$ is finite, taking the geometric generic fibers in \eqref{equDlu} and  \eqref{equDmu}, we get $$A_1\sim \sA_{s_2}\otimes_kk_1.$$
Thus $A_1$ is defined over $k$ up to isogeny.
By symmetry, $A_2$ is defined over $k$ up to isogeny.
This concludes the proof.
\endproof

\medskip

\begin{lem}\label{lemsubvardef}
Let $k_1, k_2$ be two intermediate fields of a field extension $K/k$ such that  $k_1\cap k_2=k$.
Let $Y$ be a variety over $k$. Let $X$ be a closed subvariety of $Y_K=Y\otimes_kK$ over $K$.
Assume that as a subvariety of $Y_K$, $X$ is defined over $k_1$ and also defined over $k_2$.
Then $X$ is defined over $k.$
\end{lem}

\proof
By taking an open affine cover of $Y$, it suffices to treat the case that $Y=\Spec R_0$ is affine. Denote by $I$ the ideal of $R_{0,K}=R_0\otimes_kK$ associated to $X$.
We only need to show $I=I_0\otimes_kK$ for $I_0:=I\cap R_0.$ Here we identify $R_0$ with its image in $R_{0,K}=R_0\otimes_kK.$

This converts to a basic result in linear algebra.  
Namely, let $K,k,k_1,k_2$ be as before, let $V_0$ be a vector space over $k$, and let $V=V_0\otimes_kK$ be the vector space over $K$.
Let $W$ be a $K$-subspace of $V$. 
Assume that $W$ can be descended to $k_i$ for $i=1,2$; i.e.,
$W=W_i\otimes_{k_i}K$ for a $k_i$-subspace $W_i$ of $V_0\otimes_kk_i$. 
Then $W$ can be descended to $k$. 

If $V_0$ is finite-dimensional, the result is an easy consequence of the existence of the Grassmannian variety, but we give an elementary proof as follows.
Write $n=\dim_k V_0$ and $m=\dim_K W$. 
Take a $k$-basis of $V_0$, and use it to identify $V_0=k^n$ and $V=K^n$.
Then $W$ is represented by a $m\times n$ matrix with coefficients in $K$. 
In fact, take a basis of $W$, each element of which gives a row of the matrix.
By row operations, we can convert the matrix to a reduced row echelon form $E$ over $K$. 
We can also get a reduced row echelon form $E_i$ over $k_i$ for $W_i\subset k_i^n$. 
By the uniqueness of the reduced row echelon from, we have $E=E_1=E_2$. Then the coefficients of $E$
are in $k_1\cap k_2=k$.
It follows that $W$ can be descended to $k$.

If $V_0$ is infinite-dimensional, let $V_0'$ be a finite-dimensional subspace over $k$ and write $V'=V_0'\otimes_kK$.
Apply the above results to the subspace $W\cap V'$ of $V'$. 
We see that $W\cap V'$ can be descended to $k$.
Vary $V_0'$. Note that any vector of $W$ is contained in some $V'$. 
It follows that $W$ is contained in the $K$-subspace of $V$ spanned by $W\cap V_0$. 
Then $W$ can be descended to $k$.
This concludes the proof.
\endproof

\begin{cor}\label{cordeffield}
Let $K/k$ be an extension of algebraically closed fields. 
\begin{points}
\item Let $A$ be an abelian variety over $K.$
Then there is an algebraically closed intermediate field $k_A$ of $K/k$, such that for every algebraically closed intermediate field $k'$ of $K/k$, 
$A$ is defined over $k'$ up to isogeny if and only if $k_A\subseteq k'.$
Moreover, we have $\trdeg(k_A/k)<\infty.$
\item Let $Y$ be a variety over $k$. Let $X$ be a subvariety of $Y_K=Y\otimes_kK.$
Then there is an algebraically closed intermediate field $k_{X\subseteq Y_K}$ of $K/k$, such that for every algebraically closed intermediate field $k'$ of $K/k$,
$X\subseteq Y_K$ is defined over $k'$ if and only if $k_{X\subseteq Y_K}\subseteq k'.$
Moreover, we have $\trdeg(k_{X\subseteq Y_K}/k)<\infty.$
\end{points}
\end{cor}
\proof
We only prove (i). The proof for (ii) is almost the same, except replacing  Proposition \ref{lembasechangediagram} by Lemma \ref{lemsubvardef}.

Let $I$ be the set of algebraically closed intermediate fields $k'$ of  $K/k$ such that $A$  is defined over  $k'$.
Because $A$ is of finite type, for every $k'\in I$, there is $k''\in I$ contained in $k'$ and satisfying $\trdeg(k''/k)<\infty.$

So there is $k_A\in I$, such that $\trdeg(k_A/k)$ is the smallest. 
It is clear that for every algebraically closed intermediate field $k'$ of $K/k$, if $k_A\subseteq k'$, $k'\in I.$
So we only need to show that  for every $k'\in I$, $k_A\subseteq k'.$ One may assume that $\trdeg(k'/k)<\infty.$
By Proposition \ref{lembasechangediagram}, $k'\cap k_A\subseteq I.$
This proves $k_A\subseteq k'$ because $\trdeg(k_A/k)$ is the smallest.
\endproof

\subsection{Special subvarieties of abelian varieties}\label{subsectionnonspecial}

%

Let $K/k$ be a field extension such that $k$ is algebraically closed in $K$. 
Denote by $I(K/k)$ the set of intermediate fields $k'$ of $K/k$ which is algebraically closed in $K.$
If $\trdeg(K/k)>1$, then  $I(K/k)$ is infinite.

Let $A$ be an abelian variety over $K$.   Recall that 
a subvariety $X$ of $A_{\overline{K}}$ is said to be special in $A/K/k$ if  
\begin{equation}
X=\tr(Y{\otimes_{\overline{k}}\overline{K}})+T
\end{equation}
for some torsion subvariety $T$ of $A_{\overline{K}}$ and  some subvariety $Y$ of $A^{\overline{K}/\overline{k}}$.



\begin{pro}\label{lemnonspecial}
Let $K/k$ be a field extension such that $k$ is algebraically closed in $K$. 
Let $A$ be an abelian variety over $K$.  Let $X$ be a subvariety of $A_{\overline{K}}.$
Then there is $k_{A,X}\in I(K/k)$ such that for every $k'\in I(K/k)$, $X$ is special  for $A/K/k'$  if and only if $k_{A,X}\subseteq k'$.
\end{pro}

\proof 
There is an bijection $\phi: I(\overline{K}/\overline{k})\to I(K/k)$ sending $k'$ to $k'\cap K.$ Because $\phi$ preserves the ordering and for $k'\in I(\overline{K}/\overline{k})$, and 
$X$ is special  for $A_{\overline{K}}/\overline{K}/k'$ if and only if $X$ is special  for $A/K/\phi(k')$, we may assume that $k$ and $K$ are algebraically closed.

\medskip

Let $\Stab^0(X)$ be the identity component of the closed subgroup scheme
$$\Stab(X):=\{g\in A|\,\, g+X=X\}.$$ 
We note that, for every $k'\in I(K/k)$,  $X$ is special  for $A/K/k'$ if and only if $X/\Stab^0(X)$ is special  for $(A/\Stab^0(X))/K/k',$
where  $X/\Stab^0(X)$ is the image of $X$ under the quotient morphism $A\to A/\Stab^0(X).$
After replacing $(A, X)$ by $(A/\Stab^0(X), X/\Stab^0(X))$, we may assume that $\Stab^0(X)=0.$

Let $T_X$ be the minimal torsion subvariety of $A$ containing $X.$ Let $a$ be a torsion point in $T_X$, so that $T_X-a$ is an abelian subvariety of $A$. 
Then for every $k'\in I(K/k)$,  $X$ is special  for $A/K/k'$ if and only if $X-a$ is special  for $(T_X-a)/K/k'.$ After replacing $(A,X)$ by $(T_X-a, X-a)$, we may assume that $T_X=A.$

We first assume that $\tr(A^{K/k})_K\neq A$. 
By Corollary \ref{cordeffield}(i), there is $k_{A,X}:=k_A\in I(K/k)$ such that for any $k'\in I(K/k)$,
$\tr(A^{K/k'})_K= A$ if and only if $k_{A,X}\subseteq k'.$ 
Since $\Stab^0(X)=0$ and $T_A(X)=A$, $X$ is special for $A/K/k'$ if and only if $k_{A,X}\subseteq k'.$

Now we assume that $\tr(A^{K/k})_K= A.$ Pick an isogeny $\Phi: A\to \tr(A^{K/k})_K.$ Then for every $k'\in I(K/k)$,  $X$ is special for $A/K/k'$ if and only if $\Phi(X)$ is special for $\tr(A^{K/k'})_K/K/k'.$
Moreover, we still have $\Stab^0(\Phi(X))=0$ and $T_{\Phi(X)}=\tr(A^{K/k})_K.$
After replacing $(A, X)$ by $(\tr(A^{K/k})_K, \Phi(X))$, we may assume that $A$ is defined over $k.$
In this case, for any $k'\in I(K/k)$, $X$ is special in $A/K/k'$ if and only if  $X$ as a subvariety of $A$ is defined over $k'.$
By Corollary \ref{cordeffield}(ii), there is $k_{A,X}:=k_{X\subseteq A}\in I(K/k)$ such that for every $k'\in I(K/k)$, $X$ is special in $A/K/k'$
if and only if $k_{A,X}\subseteq k'$.
\endproof

\subsection{Proof of Proposition \ref{changefield}}

We start with the following general result. 

\begin{lem}\label{generic hyperplane} 
Let $k$ be an algebraically closed field.
Let $K/k$ be a finitely generated extension of transcendence degree $\trdeg(K/k)>1.$
Let $(S,\sM)$ be a polarization of $K/k$.
Assume that $\sM$ is very ample over $S$.

Then there are infinitely many intermediate fields $k'$ of $K/k$, algebraically closed in $K$ and of transcendence degree $1$ over $k$, together with a geometrically integral subvariety $H$ in $S_{k'}$ of codimension 1 satisfying the following properties.
\begin{itemize}
\item[(1)]
the divisor $H$ of $S_{k'}$ is linearly equivalent to $\sM_{k'}$;
\item[(2)]
the composition $H\to S_{k'} \to S$ induces an isomorphism between the generic points of $H$ and $S$.
\end{itemize}

As a consequence of (1) and (2), the pair $(k',H)$ satisfies the following property.
Let $A$ be any abelian variety over $K$, and let $L$ be any symmetric and ample line bundle over $A$.
Then under the polarization $(S,\sM)$ of $K/k$ and the polarization of $(H, \sM_{k'}|_H)$ of $K/k'$, we have the inequality 
$$
\hat h_L^\sM(X)\geq \hat h_L^{\sM_{k'}|_H}(X)
$$
of canonical heights for any subvariety $X$ of $A_{\overline{K}}.$
\end{lem}

\begin{proof}
By a Bertini type of theorem (cf. \cite[Theorem 6.10]{Jouanolou}), for a general section 
$s_0\in \Gamma(S,\sM)$, $H_0=\div(s_0)$ is geometrically integral.
Set $H_1=\div(s_1)$ for some $s_1\in \Gamma(S,\sM)$ which is not a multiple of $s_0$.
Then $s_0$ and $s_1$ determine a plane in $\Gamma(S,\sM)$, and thus a pencil in $S$ with base locus $H_0\cap H_1$. 
This gives a rational map $S\dashrightarrow \P^1_k$ sending $x$ to $(s_0(x),s_1(x))$. 
Let $\pi:\widetilde S\to S$ be the blowing-up along $H_0\cap H_1$. 
Then the rational map becomes a flat morphism $\psi:\widetilde S\to \P^1_k$ as in \cite[II, Example 7.17.3]{Hartshorne1977}.

For any point $(a_0,a_1)\in \P^1_k(k)$, the fiber $\psi^{-1}(a_0,a_1)\subset \widetilde S$ is isomorphic to the hyperplane section $\div(a_0s_1-a_1s_0)$ of $S$. In fact, $\psi^{-1}(a_0,a_1)\subset \widetilde S$ is the strict transform of 
$\div(a_0s_1-a_1s_0)$ in $\widetilde S$, and thus isomorphic to the blowing-up of $\div(a_0s_1-a_1s_0)$ along $H_0\cap H_1$. Note that $H_0\cap H_1$ is a Cartier divisor in $\div(a_0s_1-a_1s_0)$, by  
$H_0\cap H_1=\div(a_0s_1-a_1s_0)\cap \div(s_0)$ if $a_1\neq 0$ and $H_0\cap H_1=\div(a_0s_1-a_1s_0)\cap \div(s_1)$ if $a_0\neq 0$. 
Then the blowing-up of $\div(a_0s_1-a_1s_0)$ along $H_0\cap H_1$ does not change $\div(a_0s_1-a_1s_0)$. 

As a consequence, $H_0$ and $H_1$ are closed fibers of $\psi$. 
As $H_0$ is geometrically integral, the generic fiber $H' \to \Spec k'$ of $\psi:\widetilde S\to \P^1_k$ is geometrically integral.
Here $k'=k(\P^1_k)$ is set to be the function field of $\P^1_k$, and identified with a subfield of $K$ via $\psi$. 
Because $H' \to \Spec k'$ is geometrically integral,
$k'$ is algebraically closed in $K.$ 

Let $H$ be the image of $H'$ under the morphism $\pi\times \psi: \widetilde S \to S\times \P^1_k.$
Let $\pi_1: S\times \P^1_k\to \P^1_k$ be the projection to the first factor and 
$\pi_2:S\times \P^1_k\to \P^1_k$ be the projection to the second factor. Then $H\subseteq \pi_2^{-1}(\Spec k')=S_{k'}.$
The generic point of $H$ is also the generic point of $\widetilde S$, so $H$
satisfies conditions (2). 

Now we check that $H$ satisfies condition (1). 
As above, for any field extension $F/k$ and any point $t\in \P^1_k(F)=\P^1_F(F)$, the fiber $\psi^{-1}(t)\subset \widetilde S_F$ is isomorphic to a hyperplane section of $S_F$ corresponding to $\sM_F$. 
Take $F=k'$ and take
$t\in \P^1_k(F)$ to be the generic point $\Spec k'\to \P^1_k$.
This implies that $H'$ is isomorphic to a hyperplane section of $S_{k'}$. 
The same result holds for $H$, since it is the image of the composition $H'\to S_{k'} \to S\times \P^1_k$.

It remains to check the height inequality. 
We can assume that $X$ is a closed subvariety of $A$. 
Let$(\sA,\sL)$ be an integral model of $(A,L)$ over $S$ with structure morphism $\pi:\sA\to S$. 
We can further assume that $\sL$ is ample, which can be achieved by passing to a tensor power of $L$. 
Denote by $\sX$ the Zariski closure of $X$ in $A$.
It suffices to prove that the inequality for the relevant naive heights, which becomes the inequality
$$
\sL^{\dim X+1}\cdot (\pi^*\sM)^{\dim S-1} \cdot \sX
\geq (\sL_H)^{\dim X+1}\cdot (\pi_H^*\sM_H)^{\dim S-2} \cdot \sX_H^*.
$$
Here the right-hand side is an intersection in $\sA_H=\sA\times_SH$, and $\pi_H,\sL_H,\sM_H, \sX_H$ denote the base change of $\pi,\sL,\sM, \sX$ via the morphism $H\to S$.
Moreover, $\sX_H^*$ denotes the Zariski closure of $X$ in $\sA_H$. 

Denote by $\pi_{k'}:\sA_{k'}\to S_{k'},\sL_{k'},\sM_{k'}, \sX_{k'}$ the base change of $\pi:\sA\to S,\sL,\sM, \sX$ via the morphism $\Spec k'\to \Spec k$.
By base change, the left-hand side of the inequality is equal to 
$$
\sL_{k'}^{\dim X+1}\cdot (\pi_{k'}^*\sM_{k'})^{\dim S-1} \cdot \sX_{k'},$$
which is an intersection in $\sA_{k'}$. 
Note that $\sM_{k'}$ is linearly equivalent to the hyperplane section $H$ of $S_{k'}$.
The intersection number is further equal to 
$$(\sL_H)^{\dim X+1}\cdot (\pi_H^*\sM_H)^{\dim S-2} \cdot (\sX_{k'})_H.$$
Here $(\sX_{k'})_H= \sX_{k'}\cap \pi_{k'}^{-1}(H)$ is a proper intersection in $\sA_{k'}$, so it is equi-dimensional. 

The Zariski closure $\sX_H^*$ of $X$ in $\sA_H$ is an irreducible component of $(\sX_{k'})_H$.
Then we have $(\sX_{k'})_H- \sX_H^*$ is an effective cycle. 
As $\sL$ and $\sM$ are ample, we have 
$$(\sL_H)^{\dim X+1}\cdot (\pi_H^*\sM_H)^{\dim S-2} \cdot (\sX_{k'})_H
\geq (\sL_H)^{\dim X+1}\cdot (\pi_H^*\sM_H)^{\dim S-2} \cdot \sX_H^*.$$
This proves the inequality.
\end{proof}

Now we are ready to prove Proposition \ref{changefield}.

\begin{proof}[Proof of Proposition \ref{changefield}]
Let $A/K/k$ be as in the proposition.
Let $L$ be an ample line bundle over $A$, and let $(S,\sM)$ be a polarization of $K/k$.
We can assume that $\sM$ is very ample by passing to a positive tensor power. 

Choose $(H_1,k_1), (H_2,k_2)$ as in Lemma \ref{generic hyperplane} such that $k_1\neq k_2$. 
For $i=1,2$, denote by $\sM_i$ the pull-back of $\sM$ via $(H_i)_{{\overline k_i}}\to S$. 
Note that $K{\overline k_i}={\overline k_i}(H_{{\overline k_i}}).$
For any $x\in A(\OK)$, we have the height inequality
$$
\widehat h_L^{\sM}(x)\geq \widehat h_L^{\sM_i}(x).
$$

Let $X$ be a subvariety of $A_{\overline{K}}$ which contains a dense set of small points of $A/K/k$.
Then $X$ contains a dense set of small points of $A_{K{\overline k_i}}/K{\overline k_i}/{\overline k_i}$
for $i=1,2$ by the height inequality. 
By assumption, the geometric Bogomolov conjecture holds for $(A_{K{\overline k_i}}/K{\overline k_i}/{\overline k_i}, X)$, so $X$ is special in $A/K/k_i$ for both $i=1,2.$ 

If $X$ is not special for $A/K/k$, let $k_{A,X}$ as in Proposition \ref{lemnonspecial}. We have $k_{A,X}\neq k.$
By $\trdeg(k_i/k)=1$, we have $k_{A,X}\not\subseteq k_i$ for at least one $i\in \{1,2\}$.
Then $X$ is special for $A/K/k_i$ for that $i$.
This is a contradiction. 
\end{proof}

\section{Line bundles over abelian schemes} \label{sec line bundle}

In this section, we introduce some preliminary results for abelian schemes over curve, which will be used in the proof the geometric Bogomolov conjecture in \S\ref{sec final}. These results are well-known, but we collect them here for convenience.

\subsection{Rigidified line bundles}
Let $S$ be a smooth projective curve over a field $k$. 
Let $\pi:\sA\to S$ be an abelian scheme over $S$.
This will be the basic setup of this section.

For any $m\in\Z$, denote by $[m]:\sA\to \sA$ the homomorphism of multiplication by $m$, and denote by $\sA[m]$ the kernel of this homomorphism. 

By a \emph{multi-section} of $\pi:\sA\to S$, we mean a closed integral subscheme $\sT$ of $\sA$ such that the induced morphism $\sT\to S$ is finite and flat. 
The multi-section $\sT$ is called \emph{torsion} if its corresponding element in the abelian group $\sA_\sT(\sT)$ is torsion. 
The \emph{order} of the torsion multi-section $\sT$ is defined to be the order of the corresponding element in $\sA_\sT(\sT)$.

Any torsion multi-section $\sT$ is necessarily the Zariski closure of a torsion point of $A(\OK)_\tor$ in $\sA$. 
If the order of $\sT$ is $m$, then $\sT$ is a closed subscheme of $\sA[m]$. 
If furthermore $m$ is not divisible by $\mathrm{char}(k)$, then $\sA[m]$ is finite and \'etale over $S$, and thus $\sT$ is also finite and \'etale over $S$. 
In this case, $\sT$ is actually smooth over $k$.

Let $\sL$ be a line bundle over $\sA$. 
We say that $\sL$ is \emph{symmetric} if $[-1]^*\sL\simeq \sL$. 
We say that $\sL$ is \emph{anti-symmetric} if $[-1]^*\sL\simeq \sL^\vee$. 
We say that $\sL$ is \emph{rigidified} if it is endowed with an isomorphism $e^*\sL\simeq \sO_S$ via the identity section $e:S\to \sA$. The isomorphism is called a \emph{rigidification}.
The following results are well-known to experts, but we sketch proofs for convenience of readers.

\begin{lemma} \label{line bundle}
Let $\sL$ be a rigidified line bundle over $\sA$. 
Then the following hold.
\begin{itemize}
\item[(1)] If $\sL$ is symmetric, then $[m]^*\sL\simeq m^2\sL$ for any $m\in \Z$; if $\sL$ is anti-symmetric, then $[m]^*\sL\simeq m\sL$ for any $m\in \Z$.
\item[(2)] For any torsion multi-section $\sT\subset \sA$, the line bundle $\sL|_{\sT}$ is torsion in $\Pic(\sT)$.
\item[(3)] If $\sL$ is symmetric and $\pi$-ample, then $\sL$ is nef over $\sA$. 
\item[(4)] If $\sL$ is symmetric and $\pi$-ample, and $\sL'$ is another symmetric and rigidified line bundle over $\sA$, then there is a positive integer $a$ such that $a\sL-\sL'$ is nef over $\sA$. 
\end{itemize}
\end{lemma}

\begin{proof}
For (1), we only consider the symmetric case as the anti-symmetric case is similar. 
Note that $[m]^*\sL- m^2\sL$ is trivial over every fiber of $\pi:\sA\to S$, so it lies in $\pi^*\Pic(S)$. Consider the pull-back to the identity section. Then $[m]^*\sL- m^2\sL$ is trivial by the rigidification. 

For (2), by $2\sL=(\sL+[-1]^*\sL)+(\sL-[-1]^*\sL)$, we can assume that $\sL$ is either symmetric case or anti-symmetric. Let $m$ be the order of $\sT$. It suffices to prove that $\sL|_{\sA[m]}$ is torsion. Note that $\sA[m]\to S$ is the base change of $[m]:\sA\to \sA$ by the identity section. 
It follows that $([m]^*\sL)|_{\sA[m]}$ is trivial over $\sA[m]$ by the rigidification. 
Then $\sL|_{\sA[m]}$ is torsion over $\sA[m]$ by (1).

For (3), since $\sL$ is $\pi$-ample, there is an ample line bundle $\sL'$ over $S$ such that $\sL+\pi^*\sL'$ is ample over $\sA$. 
Then $[m]^*(\sL+\pi^*\sL')\simeq m^2\sL+\pi^*\sL'$ is ample over $\sA$.
The $\Q$-line bundle $\sL+m^{-2}\pi^*\sL'$ is ample over $\sA$, and thus its limit $\sL$ is nef over $\sA$.

For (4), note that $a\sL-\sL'$ is $\pi$-ample for sufficiently large $a$. Apply (3). 
\end{proof}

\subsection{Numerical classes of torsion multi-sections}

The following actually holds for rational equivalence and for high dimensional base $S$ (cf. \cite[Theorem 2.1]{GZ2014}). 
We will be content to numerical equivalence for $\dim S=1$, which has the following quick proof and is sufficient for our application.

\begin{pro} \label{numerical equivalent}
Let $S$ be a smooth projective curve over a field $k$. 
Let $\pi:\sA\to S$ be an abelian scheme over $S$ of relative dimension $g\geq 1$.
Let $\sL$ be a symmetric and rigidified line bundle over $\sA$. 
Let $\sT$ be a torsion multi-section of $\sA$ over $S$. 
Then there is a numerical equivalence 
$$[\sL]^g\equiv
\frac{\deg(\sL_\eta)}{\deg(\sT/S)} [\sT]$$ 
of 1-cycles over $\sA$. 
Here $\eta$ is the generic point of $S$, and $\deg(\sL_\eta)$ is the degree of $\sL_\eta$ over the abelian variety $\sA_\eta$. 
\end{pro}

\begin{proof}
It suffices to prove that 
$$\sL^g\cdot \sM=
\frac{\deg(\sL_\eta)}{\deg(\sT/S)} \sT\cdot \sM$$ 
for any line bundle $\sM$ over $\sA$. 

We first prove that if $\sM$ is rigidified, then both sides are 0. 
The right-hand side is 0 by Lemma \ref{line bundle}(2).
For the left-hand side, by the decomposition $2\sL=(\sL+[-1]^*\sL)+(\sL-[-1]^*\sL)$ again, we can assume that $\sL$ is either symmetric case or anti-symmetric. 
Then $[m]^*\sM\simeq m^i\sM$ for $i=1,2$.
By the projection formula, 
$$([m]^*\sL)^g\cdot ([m]^*\sM)= 
\deg([m]) \sL^g\cdot \sM.$$ 
This is just 
$$m^{2g+i} \,\sL^g\cdot \sM= 
m^{2g}\, \sL^g\cdot \sM.$$ 
It follows that $\sL^g\cdot \sM=0.$

To prove the identity for general $\sM$, write $\sM=(\sM-\pi^*e^*\sM)+\pi^*e^*\sM$. Note that $\sM-\pi^*e^*\sM$ is canonically rigidified. It is reduced to prove the result for $\sM=\pi^*\sN$ for line bundles $\sN$
over $S$. This follows from the simply equalities
$$\sL^g\cdot \pi^*\sN=\deg(\sL_\eta)\deg(\sN),\quad
\sT\cdot \pi^*\sN=\deg(\sT/S)\deg(\sN).$$ 
This finishes the proof.
\end{proof}

\subsection{Canonical height}

Let $S$ be a smooth projective curve over a field $k$. 
Let $K=k(S)$ be the function field. 
Let $A$ be an abelian variety over $K$, and $L$ is a symmetric and ample line bundle over $A$.
Let $X$ be a closed subvariety of $A$. 
Then we have the canonical height $\hat h(X)$ associated to $L$. 

Assume that $A$ has everywhere \emph{good reduction} over $S$.
In this case, we have the following well-known easy interpretation of the canonical height of closed subvarieties of $A$. 

In fact, let $\pi:\sA\to S$ be the unique abelian scheme with generic fiber $A\to \Spec K$.
Let $\sL$ be a symmetric and rigidified line bundle over $\sA$ extending $L$.

For the existence of $\sL$, we first take any line bundle $\sL'$ over $\sA$ extending $L$, which can be obtained by passing to divisors and taking Zariski closures. 
By replacing $\sL'$ by $\sL'-\pi^*e^*\sL'$, we can assume that $\sL'$ is rigidified.  Here $e:S\to\sA$ is the identity section.
Then $\sL$ is automatically symmetric. 
In fact, $\sL-[-1]^*\sL$ is trivial over $A$, so it is isomorphic to $\pi^*\sM$ for some line bundle $\sM$ over $S$.
The rigidification implies that $\sM$ is trivial. 

Denote by $\sX$ the Zariski closure of $X$ in $A$. 
By Lemma \ref{line bundle}(1) and the projection formula, the naive height
$$
h_{\sL}(X)=\frac{\sL^{\dim X+1}\cdot \sX}{(\dim X+1)\deg_L(X)}
$$
already satisfies
$h_{\sL}([n]X)=n^2\, h_{\sL}(X)$. 
As a consequence, $h_{\sL}(X)$ is exactly equal to the canonical height 
$\hat h(X)$ associated to $L$.

\section{Proof of the geometric Bogomolov conjecture} \label{sec final}

The goal of this section is to prove Theorem \ref{thmmain}.
Let $A/K/k$ be as in the theorem. 
By Proposition \ref{changefield}, we can assume the following condition.
\begin{itemize}
\item[(a)] the transcendence degree of $K$ over $k$ is 1. 
\end{itemize}
Let $S$ be the unique (normal) projective model of $K/k$. 
We can further make the following assumptions.
\begin{itemize}
\item[(b)] $S$ is a smooth projective curve over $k$;
\item[(c)] $A$ has semi-stable reduction over $S$. 
\end{itemize}
The second condition is a consequence of the semistable reduction theorem. 

Our next step is to apply Yamaki's result to make the following further assumption:
\begin{itemize}
\item[(d)] $A$ has trivial $\overline K/k$-trace and good reduction over $S$. 
\end{itemize} 
Denote by $A'$ the maximal abelian subvariety of $A$ that has everywhere good reduction over $S$. 
Yamaki \cite[Theorem 1.5]{Yamaki2016a} asserts that the geometric Bogomolov conjecture holds for $A$ if and only if it holds for $A'/\tr(A^{K/k}\otimes_kK)$. 
This gives the condition. 

The major contribution of this section is to prove the geometric Bogomolov conjecture
under (a)-(d), based on the Manin--Mumford conjecture in this case. 

Recall that the Manin--Mumford conjecture was proved by Raynaud \cite{Raynaud1983,Raynaud1983a} over number fields, and proved by Hrushovski \cite{Hrushovski2001} over arbitrary fields.  Hrushovski's proof relies in the model theory of difference fields. Inspired by Hrushovski's proof, Pink--R\"ossler \cite{Pink2002a,Pink2004} gave a new proof using classical algebraic geometry. We will only need the conjecture under assumption (a)-(d). 
For convenience, state it in the following case of trivial $\overline K/k$-trace.

\begin{thm}[Manin--Mumford conjecture] \label{MM}
Let $K$ be a finitely generated field over an algebraically closed field $k$. Let $A$ be an abelian variety over $K$ of trivial $\overline K/ k$-trace.
Let $X$ be a closed subvariety of $A_{\overline K}$. 
Assume that $X(\overline K)\cap A(\overline K)_\tor$ is Zariski dense in $X$. 
Then $X$ is a torsion subvariety of $A_{\overline K}$. 
\end{thm}

We also need the following well-known consequence of Zhang's fundamental inequality.

\begin{lem}\label{fundamental inequality} 
Let $K/k$ be a finitely generated field extension of transcendence degree $1$. 
Let $A$ be an abelian variety over $K$. Let $L$ be a symmetric and ample line bundle over $A$.
Let $X$ be a closed subvariety of $A_\OK$.
If $X$ contains a dense set of small points of $A/K/k$, then $\widehat h_L(X)=0$.
\end{lem}
\begin{proof}
It is a direct consequence of the fundamental inequality, which asserts that
$$
\sup_{U\subset X} \inf_{x\in U(\overline K)}  
\widehat h_L(x)
\geq   \widehat h_L(X).$$
Here the supremum goes through all open subvarieties $U$ of $X$.
The fundamental inequality was proved over number fields by Zhang \cite[Theorem 1.10]{Zhang1995}, based on his previous works \cite{Zhang1992, Zhang1995a} and as a part of this theorem of successive minima. 
It was generalized to function fields by Gubler \cite[Lemma 4.1]{Gubler2007}. 
\end{proof}

\subsection{Subvarieties generated by addition}

With these preparations, we are ready to prove Theorem \ref{thmmain}. 
Let $A/K/k$ and $X$ be as in the theorem. 
By the above reduction, we can assume that conditions (a)-(d) hold.
By extending $K$ and $k$ if necessary, we can assume that $X=X^*\otimes_K \overline K$ for a subvariety $X^*$ of $A$ over $K$. 
By abuse of notations, we will write $X=X^*$, viewed as a subvariety of $A$.  

In this step, we consider the sum $X$ with itself in $A$. The key assumption we are going to use is that $A$ has a trivial $\overline K/k$-trace. 

For any integer $m\geq1$, denote by $X_m$ the image of the addition morphism 
$$f_m:X^m\longrightarrow A, \quad 
(x_1,\cdots, x_m) \longmapsto x_1+ \cdots+ x_m.$$ 
Set $X_0=0$ to be the identity point of $A$.
Since $X_m$ is the image of the addition morphism
$X_{m-1}\times X\to A$, we have
$$\dim X_{m-1}\leq \dim X_m\leq \dim X_{m-1}+\dim X, \quad m\geq1.$$ 

\begin{lem} \label{addition}
There is a unique integer $r\geq 1$ such that $\dim X_{r-1}<\dim X_r$ and $X_r$ is a torsion subvariety of $A$.
\end{lem}
\begin{proof}
Fix a point $x_0\in X(K)$, which exists by replacing $K$ by a finite extension. 
Denote by 
$X_m'=X_m-mx_0$
the translation of $X_m$ by $-mx_0$ in $A$.
Note that $X_m'$ is an (irreducible) subvariety of $A$, and that the sequence 
$\{X_m'\}_m$ is increasing. 
As a consequence, there is $r\geq 1$ such that $X_m'\subsetneq X_r'$ for all $m<r$ and $X_m'= X_r'$ for all $m\geq r$.

It follows that there is a morphism $\sigma: X_r' \times X_r' \to X_{2r}' = X_r'$ induced by the addition morphism of $A$. This is sufficient to imply that $B=X_r'$ is an abelian subvariety of $A$. 
In fact, the relative dimension of $\sigma$ is equal to $\dim B$. 
By semi-continuity, the dimension of the fiber $\sigma^{-1}(0):=\{(x,y)\in B^2:x+y=0\}$ is at least $\dim B$. 
On the other hand, the first projection $\sigma^{-1}(0)\to B$ is injective on $\OK$-points, and thus bijective on $\OK$-points by comparing dimensions. 
This implies that for any $x\in B(\OK)$, $y=-x\in B(\OK)$. 
Thus we have inverse morphism $[-1]B \subset B$. 
It follows that $B$ is an abelian subvariety of $A$. 

As a consequence, $X_r=B+t$ for the abelian subvariety $B$ of $A$ and the point $t=rx_0\in A(K)$. 
Denote $C=A/B$, which is an abelian variety over $K$.
It suffices to prove that the image $t'$ of $t$ in $C(K)$ is a torsion point.
There is a surjective homomorphism $A \to B$ such that the composition $B\to A \to B$ is an isogeny. 
This induces an isogeny $A\to A'$ with $A'=B\times C$.

By assumption, $X$ contains a dense set of small points of $A/K/k$.
Since the canonical height is quadratic and positive definite up to torsion, we have 
$$
\hat h(x_1+ \cdots+ x_m)\leq  m(\hat h(x_1)+ \cdots +\hat h(x_m)), \quad x_i\in A(\OK).
$$
Then ``small points'' of $X^r$ transfer to ``small points'' of $X_r$.
As a consequence, $X_r$ contains a dense set of small points of $A/K/k$.
Then its image $X'=B+t'$ in $A'$ contains a dense set of small points of $A'/K/k$.

Take symmetric and ample line bundle $L_1$ over $B$ and $L_2$ over $C$.
Choose $L=p_1^*L_1+p_2^*L_2$, which is a symmetric and ample line bundle over $A'$, where $p_1:B\times C\to B$ and $p_2:B\times C\to B$ are the projections. 
For any point $x\in B(\OK)$, we have the heights
$$
\hat h_L(x+t')= \hat h_{L_1}(x)+\hat h_{L_2}(t') \geq \hat h_{L_2}(t') \geq 0. 
$$
This forces $\hat h_{L_2}(t')=0$ as $B+t'$ contains a dense set of small points of $A'/K/k$.

By assumption, $A$ has a trivial $\overline K/ k$-trace, so 
$B$ and $C$ have trivial $\overline K/ k$-traces.
Then $\hat h_{L_2}(t')=0$ implies that $t'\in C(K)_\tor$. 
This follows from the Northcott property (cf. \cite[Theorem 9.15]{Conrad}).
It finishes the proof.
\end{proof}

By the lemma, $X_r$ is a torsion subvariety of $A$. Replacing $X$ by a translation by a suitable torsion point, we can assume that $X_r$ is an abelian subvariety of $A$. In this process, to make $X$ to be defined over $K$, we may need to replacing $K$ by a finite extension again.
To summarize, we write this process as the following condition. 
\begin{itemize}
\item[(e)] $A'=X_r$ is an abelian subvariety of $A$ and $\dim X_{r-1}<\dim X_r$. 
\end{itemize}

\subsection{Fibers of the addition map}

By assumption (e), the image of the summation map $f_r:X^r\to A$
is an abelian subvariety $A'$ of $A$. 
This induces a summation morphism 
$$f:X_{r-1}\times X\longrightarrow A'.$$
Denote by 
$$e=\dim X_{r-1}+\dim X-\dim A'$$ 
the relative dimension of this morphism. 
By assumption (e), we have
$$e<\dim X.$$
Denote 
$$
A'_{e+1}=\{y\in A': \dim f^{-1}(y)\geq e+1 \}.
$$
Then $A'_{e+1}$ is a closed subset of $A'$ of codimension at least 2. 
Note that $X_{r-1}\times X$ is a natural subvariety of the abelian variety $A^2=A\times A$. 
The goal of this step is to prove the following result.

\begin{pro} \label{torsion fiber}
For any $t\in A'(\OK)_\tor\setminus A'_{e+1}(\OK)$, every irreducible component of the fiber $f^{-1}(t)\subset (X_{r-1}\times X)_\OK$ has canonical height 0 in $(A\times A)_\OK$.  
\end{pro}
\begin{proof}
Denote by $B$ the preimage of $A'$ under the summation map $A\times A\to A$. 
Since the preimage of $0$ under $A\times A\to A$ is geometrically integral, $B$ is also geometrically integral, and thus an abelian subvariety of $A\times A$ over $K$. We have the following commutative diagram.
$$
\xymatrix{
&\quad X_{r-1}\times X \quad  \ar[r] \ar[rd]^f  & \quad B \quad  \ar[d]^h \ar[r] &
\quad A\times A \quad \ar[d]
\\
&
&
\quad A' \quad \ar[r]
&
\quad A \quad
}
$$

Let $S/k$ with $K=k(S)$ be as in assumption (b). 
Denote by $\pi:\sA\to S$ the unique abelian scheme extending the abelian variety $A\to \Spec K$, which exists by assumption (d). 
Then $\sA\times_S\sA \to S$ is the unique abelian scheme extending $A\times A\to \Spec K$.
Denote by $\sX$ the Zariski closure of $X$ in $A$. 
Denote by $\sA'$ (resp. $\sB$) the Zariski closure of $A'$ in $\sA$ (resp. $B$ in $\sA\times_S\sA$). Both $\sA'$ and $\sB$ are abelian schemes over $S$. 
Denote by $\sY$ is the Zariski closure of $Y=X_{r-1}\times X$ in $\sA\times_S \sA$. 
This gives an integral version of the above diagram.
$$
\xymatrix{
\quad f^{-1}(\sT) \quad \ar[r] \ar[rd]
&\quad \sY \quad  \ar[r] \ar[rd]^f  & \quad \sB \quad  \ar[d]^h \ar[r] &
\quad \sA\times_S\sA \quad \ar[d]
\\
&
\quad \sT \quad \ar[r]
&
\quad \sA' \quad \ar[r]
&
\quad \sA \quad
}
$$

Let $\sT$ be a torsion multi-section of $\sA'\to S$ of order non-divisible by $\Char K$.
Then $\sT\to S$ is finite and \'etale, and $\sT$ is a smooth projective curve over $k$.
Assume that $T=\sT_K$ is not contained in $A'_{e+1}.$
Apply Proposition \ref{prostrleqtot} to the triangle of the second diagram. 
We obtain that 
$$\sum_{i=1}^n m_i[\sZ_i]\leq h^*[\sT]\cdot [\sY]$$
in $\CH_{e+1}(\sB)_\Q$. 
Here $\sZ_1, \dots, \sZ_n$ are irreducible components of $f^{-1}(\sT)$ satisfying 
$f(\sZ_i)=\sT$, and $m_i$ is the multiplicity of $\sZ_i$ in $f^{-1}(\sT).$ 

Let $\sL_{\sA'}$ (resp. $\sL_\sB$) be a symmetric, relatively ample and rigidified line bundle over $\sA'$ (resp. $\sB$).
By Proposition \ref{numerical equivalent}, there is a numerical equivalence 
$$[\sT]\equiv
a\, [\sL_{\sA'}]^{\dim A'}$$ 
in $\CH_{1}(\sA')_\Q$ for some $a>0$.

By these two results, we have
$$\sum_{i=1}^n m_i[\sZ_i]\cdot [\sL_\sB]^{e+1}
\leq a\,  h^*[\sT]\cdot [\sY] \cdot [\sL_\sB]^{e+1}
=a\, [h^*\sL_{\sA'}]^{\dim A'} \cdot [\sL_\sB]^{e+1}\cdot [\sY] .$$
By Lemma \ref{line bundle}(4), there is a constant $b>0$ such that $b\, \sL_\sB-h^*\sL_{\sA'}$ is a nef line bundle over $\sB$.
As a consequence, 
$$[h^*\sL_{\sA'}]^{\dim A'} \cdot [\sL_\sB]^{e+1} \cdot [\sY]
\leq 
b^{\dim A'}\,  [\sL_\sB]^{\dim \sY} \cdot [\sY].
$$

As in the proof of Lemma \ref{addition},  ``small points'' of $X$ transfer to ``small points'' of $Y=X_{r-1}\times X$.
Then $Y$ has a dense set of small points in $B$.
By Lemma \ref{fundamental inequality}, the height
$h_{\sL_\sB}(Y)=0$. 
Then we have $ [\sY] \cdot [\sL_\sB]^{\dim \sY}=0$. 
This forces $[\sZ_i]\cdot [\sL_\sB]^{e+1}$, and thus $h_{\sL_\sB}(Z_i)=0$.
Here $Z_i=\sZ_{i,K}$ for $i=1, \dots, n$ are exactly the irreducible components of $f^{-1}(\sT_K)$. 
This finishes the proof.
\end{proof}

\subsection{Lowering the dimension}

Now we prove that $X$ is torsion by induction on $\dim X$.
If $\dim X=0$, we must have $\hat h(X)=0$, and then $X$ is a torsion point of $A$ by the Northcott theorem (cf. \cite[Theorem 9.15]{Conrad}).  

If $\dim X>0$, consider the morphism $f:X_{r-1}\times X\to A'$ obtained in assumption (e). 
By Proposition \ref{torsion fiber}, for any $t\in A'(\OK)_\tor\setminus A'_{e+1}(\OK)$, every irreducible component of the fiber $f^{-1}(t)\subset (X_{r-1}\times X)_\OK$ has canonical height 0 in $A\times A$.  
Note that every irreducible component of $f^{-1}(t)$ has dimension $e<\dim X$.
By induction, it is torsion in $A\times A$. 
Therefore, $f^{-1}(t)$ contains a Zariski dense set of torsion points of $A\times A$. 
As $A'(\OK)_\tor\setminus A'_{e+1}(\OK)$ is Zariski dense in $A'$,
$X_{r-1}\times X$ contains a Zariski dense set of torsion points of $A\times A$.  

By the Manin--Mumford conjecture (cf. Theorem \ref{MM}), 
$X_{r-1}\times X$ is a torsion subvariety of $A\times A$.  
Then $X$ is a torsion subvariety of $A$, since it is the image of the composition of $X_{r-1}\times X\to A\times A$ with $p_2:A\times A\to A$. 
This finishes the proof of Theorem \ref{thmmain}.


\bibliography{dd}

\end{document}